\documentclass[12pt]{amsart}

\usepackage{amsfonts,amsmath,amssymb,amsthm}
\usepackage[latin1]{inputenc}

\newcommand{\Qq}{\mathbb{Q}}

\renewcommand{\epsilon}{\varepsilon}
\renewcommand{\le}{\leqslant}
\renewcommand{\ge}{\geqslant}

\newcommand{\defi}[1]{\emph{#1}}
\DeclareMathOperator{\coeff}{coeff}
\DeclareMathOperator{\charac}{char}

{\theoremstyle{plain}
\newtheorem{theorem}{Theorem}            
      
\newtheorem{proposition}[theorem]{Proposition}  
\newtheorem{corollary}[theorem]{Corollary}  
\newtheorem{algorithm}[theorem]{Algorithm} 
}
{\theoremstyle{remark}
\newtheorem*{remark*}{Remark}  
\newtheorem*{example*}{Example}
}

\begin{document}

\title{Decomposition of polynomials and approximate roots}

\author{Arnaud Bodin}
\address{Laboratoire Paul Painlev\'e, Math\'ematiques,
Universit\'e de Lille 1,  59655 Villeneuve d'Ascq, France.}
\email{Arnaud.Bodin@math.univ-lille1.fr}
 
\begin{abstract}
We state a kind of Euclidian division theorem:
given a polynomial $P(x)$ and  a divisor $d$ of the degree of $P$,
there exist polynomials $h(x),Q(x),R(x)$ such that
$P(x) = h\circ Q(x) +R(x)$, with $\deg h=d$. Under some conditions $h,Q,R$ are unique, 
and $Q$ is the approximate $d$-root of $P$.
Moreover we give an algorithm to compute such a decomposition.
We apply these results to decide whether a polynomial in one or several variables
is decomposable or not.
\end{abstract}

\maketitle


\section{Introduction} 
\label{sec:intro}

Let  $A$ be an integral domain (i.e.~a unitary commutative ring without zero divisors).
Our main result is:
\begin{theorem}
\label{th:decomp}
Let $P \in A[x]$ be a monic polynomial.
Let $d \ge 2$ such that $d$ is a divisor of $\deg P$ and $d$ is invertible in $A$. 
There exist $h, Q, R \in A[x]$ such that
$$P(x) = h\circ Q(x) + R(x)$$
with the conditions that
\begin{enumerate}
  \item\label{it:i}     $h, Q$ are monic;
  \item\label{it:ii}   $\deg h = d$, $\coeff(h,x^{d-1})=0$, $\deg R < \deg P - \frac{\deg P}{d}$;
  \item\label{it:iii} $R(x) = \sum_{i} r_i x^i$ with $(\deg Q | i \Rightarrow r_i =0)$. 
\end{enumerate}
Moreover such $h, Q, R$ are unique.
\end{theorem}

The previous theorem has a formulation similar to the Euclidian division; but here $Q$ 
is not given (only its degree is fixed); there is a natural $Q$ (that we will compute, see 
Corollary~\ref{cor:root}) associated to $P$ and $d$.
Notice also that the decomposition $P(x) = h\circ Q(x) + R(x)$ is \emph{not} the $Q$-adic decomposition,
since the coefficients before the powers $Q^i(x)$ belong to $A$ and not to $A[x]$.

\begin{example*}
Let $P(x)= x^6+6x^5+6x+1 \in \Qq[x]$.
If $d=6$ we find the following decomposition $P(x)= h\circ Q(x)+R(x)$ with $h(x)= x^6-15x^4+40x^3-45x^2+30x-10$, $Q(x) = x+1$
 and $R(x)=0$.
If $d=3$ we have $h(x)= x^3+65$, $Q(x)=x^2+2x-4$  and $R(x)= 40x^3 - 90x$.
If $d=2$ we get $h(x) = x^2 -\frac{725}{4}$, $Q(x) = x^3+3x^2-\frac{9}{2}x+\frac{27}{2}$
 and $R(x) = -\frac {405}{4}x^2+\frac{255}{2}x$.
\end{example*}
Theorem \ref{th:decomp} will be of special interest when then ring $A$ is itself a polynomial ring. For 
instance at the end of the paper we give an example of a decomposition of a polynomial in two variables 
$P(x,y) \in A[x]$ for $A = K[y]$.

\bigskip

The polynomial $Q$ that appears in the decomposition
has  already been introduced in a rather different context.
We denote by $\sqrt[d]{P}$ the approximate $d$-root of $P$.
It is the polynomial such that $(\sqrt[d]{P})^d$ approximate
$P$ in a best way, that is to say $P-(\sqrt[d]{P})^d$ has smallest possible degree. 
The precise definition will be given in section~\ref{sec:unicity}, but we already notice 
the following:
\begin{corollary}
\label{cor:root}
$$Q = \sqrt[d]{P}$$
\end{corollary}

\bigskip

We apply these results to another situation.
Let $A = K$ be a field and $d\ge 2$. $P \in K[x]$ is said to be \defi{$d$-decomposable} in $K[x]$ if there exist
$h,Q \in K[x]$, with $\deg h = d$ such that
$$P(x) = h \circ Q(x).$$

\begin{corollary}
\label{cor:decomp}
Let $A=K$ be a field. Suppose that $\charac K$ does not divide $d$.
$P$ is $d$-decomposable in $K[x]$ if and only if $R = 0$ in the decomposition of Theorem~\ref{th:decomp}.
\end{corollary}
In particular, if $P$ is $d$-decomposable, then $P = h \circ Q$ with $Q = \sqrt[d]{P}$.

\bigskip

After the first version of this paper, M.~Ayad and G.~Ch\`eze communicated us some references
so that we can picture a part of history of the subject. Approximate roots appeared (for $d=2$) in some
work of E.D.~Rainville \cite{Ra} to find polynomial solutions of some Riccati type differential equations. 
An approximate root was seen as the polynomial part of the expansion of $P(x)^\frac{1}{d}$ into
decreasing powers of $x$. The use of approximate roots culminated with  S.S.~Abhyankar and T.T.~Moh
who proved the so-called Abhyankar-Moh-Suzuki theorem in \cite{AM1} and \cite{AM2}.
For the latest subject we refer the reader to an excellent expository article of P.~Popescu-Pampu  \cite{PP}.
On the other hand Ritt's decompostion theorems (see \cite{Sc} for example) have led to several
practical algorithms to decompose polynomials in one variable into the form $P(x)=h \circ Q(x)$: for example D.~Kozen and S.~Landau in \cite{KL} give an algorithm (refined in \cite{vzG}) that computes a decomposition in polynomial time.
Unification of both subjects starts with P.R.~Lazov and A.F.~Beardon (\cite{La}, \cite{Be}) for polynomials in one variable over complex numbers: they notice that the polynomial
$Q$ is in fact the approximate $d$-root of $P$.

\bigskip

We define approximate roots in section \ref{sec:unicity} and prove uniqueness of the decomposition of Theorem~\ref{th:decomp}.
Then in section \ref{sec:exists} we prove the existence of such decomposition and give an algorithm to compute it.
Finally in section \ref{sec:one} we apply these results to decomposable polynomials in one variable and in section \ref{sec:multi}  to decomposable polynomials in several variables.

\section{Approximate roots and proof of the uniqueness}
\label{sec:unicity}

The approximate roots of a polynomial are defined by the following property, \cite{AM1}, \cite[Proposition 3.1]{PP}.
\begin{proposition}
\label{prop:root}
Let $P \in A[x]$ a monic polynomial
and $d \ge 2$ such that $d$ is a divisor of $\deg P$ and $d$ is invertible in $A$. 
There exists a unique monic polynomial
$Q \in A[x]$ such that:
$$\deg (P-Q^d) < \deg P - \frac{\deg P}{d}.$$
\end{proposition}
We call $Q$ the \defi{approximate $d$-root} of $P$
and denote it  by $\sqrt[d]{P}$.

Let us recall the proof from \cite{PP}.
\begin{proof}
Write $P(x) = x^n+a_1x^{n-1}+a_2x^{n-2}+\ldots+a_n$
and we search an equation for $Q(x)= x^{\frac{n}{d}}+b_1x^{\frac{n}{d}-1}+ b_2 x^{\frac{n}{d}-2}+\cdots+b_{\frac{n}{d}}$.
We want $\deg (P-Q^d) <  \deg P - \frac{\deg P}{d}$,
that is to say, the coefficients of $x^n,x^{n-1},\ldots,x^{n-\frac{n}{d}}$ in $P-Q^d$ equal zero.
By expanding $Q^d$ we get the following system of equations:
\begin{equation*}
\label{eq:sys}
\begin{cases}
  a_1 = d b_1 \\
  a_2 = d b_2+ \binom{d}{2} b_1^2 \\
  \vdots \\
  a_k = d b_k + {\hspace*{-1em} \displaystyle{\sum_{i_1+2i_2+\cdots+(k-1)i_{k-1}=k}\hspace*{-1em} c_{i_1\ldots i_{k-1}} b_1^{i_1}\cdots b_{k-1}^{i_{k-1}}}}, \qquad 1 \le k \le \frac{n}{d}
\end{cases}
\tag{$\mathcal{S}$}
\end{equation*}
where the coefficients $c_{i_1\ldots i_{k-1}}$ are the multinomial coefficients defined by the following formula:
$$c_{i_1\ldots i_{k-1}} = \binom{d}{i_1,\ldots,i_{k-1}} = \frac{d!}{i_1! \cdots i_{k-1}!(d-i_1-\cdots-i_{k-1})!}.$$
The system (\ref{eq:sys}) being a triangular system, we can inductively compute
the $b_i$ for $i=1,2,\ldots,\frac{n}{d}$:
$b_1 = \frac{a_1}{d}$, $b_2 = \frac{a_2-\binom{d}{2} b_1^2}{d}$, \ldots
Hence the system (\ref{eq:sys}) admits one and only one solution $b_1,b_2,\ldots,b_{\frac{n}{d}}$.

Notice that we need $d$ to be invertible in $A$ to compute $b_i$. 
Moreover $b_i$ depends only on the first coefficients
$a_1,a_2,\ldots,a_{\frac{n}{d}}$ of $P$.
\end{proof}

Proposition \ref{prop:root} enables us to prove Corollary  \ref{cor:root}:
by condition (\ref{it:ii}) of Theorem \ref{th:decomp}
we know that $\deg (P - Q^d) <   \deg P - \frac{\deg P}{d}$
so that $Q$ is the approximate $d$-root of $P$. 
Another way to compute $\sqrt[d]{P}$ is to use iterations of Tschirnhausen transformation,
see \cite{AM1} or \cite[Proposition 6.3]{PP}.
We end this section by proving uniqueness of the decomposition of Theorem \ref{th:decomp}.
\begin{proof}
$Q$ is the approximate $d$-root of $P$ so is unique (see Proposition \ref{prop:root} above).
In order to prove the uniqueness of $h$ and $R$, we argue by contradiction. 
Suppose  $h\circ Q+R=h' \circ Q +R'$ with $R\neq R'$;
set $r_ix^i$ to be the highest monomial of $R(x)-R'(x)$. From one hand
$x^i$ is a monomial of $R$ or $R'$, hence $\deg Q \nmid i$ by condition (\ref{it:iii}) of Theorem \ref{th:decomp}.
From the equality $(h'-h)\circ Q = R-R'$ we deduce that $i = \deg (R-R')$ is a multiple
of $\deg Q$ ; that yields a contradiction. Therefore $R=R'$, hence $h=h'$.
\end{proof}

\section{Algorithm and proof of the existence}
\label{sec:exists}

Here is an algorithm to compute the decomposition of Theorem \ref{th:decomp}.
\begin{algorithm}
\label{algo}\ 
\begin{itemize}
  \item \textbf{Input.} $P \in A[x]$, $d | \deg P$.
  \item \textbf{Output.} $h, Q, R \in A[x]$ such that $P = h \circ Q + R$.
  \item \textbf{1st step.} Compute $Q = \sqrt[d]{P}$ by solving the triangular system (\ref{eq:sys})
of Proposition \ref{prop:root}.
Set $h_1(x) = x^d$, $R_1(x)=0$.
  \item \textbf{2nd step.} Compute $P_2 = P - Q^d  = P - h_1(Q) - R_1$. Look for its highest monomial $a_ix^i$.
If $\deg Q | i$ then set $h_2(x) = h_1(x) + a_i x^{\frac{i}{\deg Q}}$, $R_2 = R_1$.
If $\deg Q \nmid i$ then $R_2(x) = R_1(x) + a_ix^i$, $h_2 = h_1$.
  \item \textbf{3thd step.}  Set $P_3 = P - h_2(Q)-R_2$, look for its highest monomial $a_ix^i$,\ldots
  \item \ldots
  \item \textbf{Final step.} $P_n = P - h_{n-1}(Q)-R_{n-1} = 0$ yields the decomposition
$P = h \circ Q + R$ with $h = h_{n-1}$ and $R=R_{n-1}$.
\end{itemize}
\end{algorithm}

The algorithm terminates because the degree of the $P_i$ decreases at each step.
It yields a decomposition $P = h\circ Q +R$ that verifies all the conditions
of Theorem \ref{th:decomp}: in the second step of the algorithm, and due to Proposition \ref{prop:root}
we know that $i < \deg P - \frac{\deg P}{d}$.
That implies $\coeff(h_2,x^{d-1})=0$ and $\deg R_2 < \deg P - \frac{\deg P}{d}$. 
Therefore at the end $\coeff(h,x^{d-1})=0$.
Of course the algorithm proves the existence of the decomposition in Theorem \ref{th:decomp}.

\section{Decomposable polynomials in one variable}
\label{sec:one}

Let $K$ be a field and $d\ge 2$. $P\in K[x]$ is said to be \emph{$d$-decomposable} in $K[x]$ if there exist
$h,Q \in K[x]$, with $\deg h = d$ such that
$$P(x) = h \circ Q(x).$$
We refer to \cite{BDN} for references and recent results on decomposable polynomials 
in one and several variables.

\begin{proposition}
\label{prop:one}
Let $A=K$ be a field whose characteristic does not divide $d$.
A monic polynomial $P$ is $d$-decomposable in $K[x]$ if and only if $R = 0$ in the decomposition $P = h \circ Q+R$.
\end{proposition}

In view of Algorithm \ref{algo} we also get an algorithm to decide
whether a polynomial is decomposable or not and in the positive case give
its decomposition.   
 
\begin{proof}
If $R=0$ then $P$ is $d$-decomposable.
Conversly if $P$ is $d$-decomposable, then there exist $h,Q \in K[x]$ such that
$P = h(Q)$. As $P$ is monic we can suppose $h, Q$ monic. 
Moreover, up to a linear change of coordinates $x \mapsto x+ \alpha$,
we can suppose that $\coeff(h,x^{d-1})=0$. Therefore
$P = h(Q)$ is a decomposition that verifies the conditions of Theorem \ref{th:decomp}.
\end{proof}

\begin{remark*}
Let $P(x)=x^n+a_1x^{n-1}+\cdots+a_n$, we first consider $a_1,\ldots,a_n$
as indeterminates (i.e.~$P$ is seen as an element of $K(a_1,\ldots,a_n)[x]$).
The coefficients of $h(x), Q(x)$ and $R(x) = r_0 x^k + r_1 x^{k-1} + \cdots + r_k$ (computed by Proposition~\ref{prop:root}, the system~(\ref{eq:sys})
and Algorithm~\ref{algo}) are polynomials in the $a_i$,
in particular $r_i = r_i(a_1,\ldots,a_n) \in K[a_1,\ldots,a_n]$, $i=0,\ldots,k$.

Now we consider $a_1^*,\ldots,a_n^* \in K$ as specializations of $a_1,\ldots,a_n$
and denote by $P^*$ the specialization of $P$ at $a_1^*,\ldots,a_n^*$.
Then, by Proposition~\ref{prop:one},  $P^*$ is $d$-decomposable in $K[x]$ if and only if 
$r_i(a_1^*,\ldots,a_n^*)=0$ for all $i=0,\ldots,k$.
It expresses the set of $d$-decomposable monic polynomials of degree $n$
as an affine algebraic variety. We give explicit equations in the following example.
\end{remark*}

\begin{example*}
Let $K$ be a field of characteristic different from $2$. Let
$P(x) = x^6+a_1x^5+a_2x^4+a_3x^3+a_4x^2+ a_5x+a_6$ be a monic polynomial
of degree $6$ in $K[x]$ (the $a_i \in K$ being indeterminates).
Let $d=2$. 
We first look for the approximate $2$-root of $P(x)$.
$\sqrt[2]{P(x)} = Q(x) = x^3+b_1x^2+b_2x + b_3$.
In view of the triangular system (\ref{eq:sys})
we get
$$ b_1 = \frac{a_1}{2},\quad b_2 = \frac{a_2-b_1^2}{2}, \quad b_3 = \frac{a_3-2b_1b_2}{2}.$$
Once we have computed $Q$, we get $h(x)=x^2+a_6-b_ 3^2$.
Therefore
$$R(x)= (a_4-2b_1b_3-b_2^2)x^2+(a_5-2b_2b_3)x.$$
Now $P(x)$ is $2$-decomposable in $K[x]$
if and only if $R(x) =0$ in $K[x]$ that is to say
if and only if $(a_1,\ldots,a_6)$ satifies the polynomial system of equations
in $a_1,\ldots,a_5$:
$$
\begin{cases}
  a_4-2b_1b_3-b_2^2 = 0, \\
  a_5-2b_2b_3 = 0. \\
\end{cases}
$$
\end{example*}

\section{Decomposable polynomials in several variables}
\label{sec:multi}

Again $K$ is a field and $d \ge 2$. Set $n\ge 2$. $P \in K[x_1,\ldots,x_n]$ is said to be \defi{$d$-decomposable} in $K[x_1,\ldots,x_n]$ if there exist
$Q \in K[x_1,\ldots,x_n]$, and $h\in K[t]$  with $\deg h = d$,  such that
$$P(x_1,\ldots,x_n) = h \circ Q(x_1,\ldots,x_n).$$

\begin{proposition}
\label{prop:multi}
Let $A = K[x_2,\ldots,x_n]$, $P\in A[x_1] = K[x_1,\ldots,x_n]$ monic in $x_1$.
Fix $d$ that divides $\deg_{x_1} P$, such that $\charac K$ does not divide $d$.
$P$ is $d$-decomposable in $K[x_1,\ldots,x_n]$ if and only if the decomposition 
$P = h \circ Q +R$ of Theorem \ref{th:decomp} in $A[x_1]$
verifies $R=0$ and $h \in K[t]$ (instead of $h \in K[t,x_2,\ldots,x_n]$).
\end{proposition}

\begin{proof}
If $P$ admits a decomposition as in Theorem \ref{th:decomp} with
$R=0$ and $h \in K[t]$ then $P = h \circ Q$ is $d$-decomposable.

Conversly if $P$ is $d$-decomposable in
$K[x_1,\ldots,x_n]$ then $P = h \circ Q$ with $h\in K[t]$, $Q \in K[x_1,\ldots,x_n]$.
As $P$ is monic in $x_1$ we may suppose that $h$ is monic and $Q$ is monic in $x_1$.
We can also suppose $\coeff(h,t^{d-1})=0$.
Therefore  $h$, $Q$ and $R:=0$ verify the conditions of Theorem \ref{th:decomp} in $A[x]$.
As such a decomposition is unique, it ends the proof.
\end{proof}

\begin{example*}
Set $A = K[y]$ and let $P(x) = x^6+a_1x^5+a_2x^4+a_3x^3+a_4x^2+ a_5x+a_6$ be a monic polynomial
of degree $6$ in $A[x]=K[x,y]$, with coefficients $a_i = a_i(y) \in A = K[y]$.
In the example of section \ref{sec:one} we have computed
the decomposition $P = h \circ Q + R$ for $d=2$
and set $ b_1 = \frac{a_1}{2},\quad b_2 = \frac{a_2-b_1^2}{2}, \quad b_3 = \frac{a_3-2b_1b_2}{2}.$
We found $h(t)= t^2+a_6-b_ 3^2 \in A[t]$
and $R(x) = (a_4-2b_1b_3-b_2^2)x^2+(a_5-2b_2b_3)x \in A[x].$
By Proposition~\ref{prop:multi} above, we get that 
$P$ is $2$-decomposable in $K[x,y]$ if and only
$$
\begin{cases}
  a_6-b_ 3^2 \in K, \\
  a_4-2b_1b_3-b_2^2 = 0  \quad \text{ in }K[y], \\
  a_5-2b_2b_3 = 0 \quad \text{ in }K[y]. \\
\end{cases}
$$
Each line yields a system of polynomial equations in 
the coefficients $a_{ij} \in K$ of $P(x,y)=\sum a_{ij}x^iy^j \in K[x,y]$.
In particular the set of $2$-decomposable monic polynomials of degree $6$ in $K[x,y]$ is an affine algebraic variety.
\end{example*}


\end{document}